\documentclass[11pt]{amsart}
\usepackage{amsmath, amssymb, amscd, amsthm, mathrsfs, url ,pinlabel,hyperref, verbatim, lipsum, wrapfig}
\usepackage[text={6.5in,9.5in}, centering, letterpaper, dvips]{geometry}
\usepackage{color,multirow,dcpic,latexsym,pictexwd,graphicx,pdfpages}
\usepackage{fancyhdr, tikz, setspace, tensor, enumitem, dirtytalk}
\usetikzlibrary{positioning, arrows.meta}
\usepackage{scrextend}

\definecolor{maroon}{rgb}{0.5, 0.0, 0.0}

\hypersetup{
  colorlinks   = true, 
  urlcolor     = maroon, 
  linkcolor    = maroon, 
  citecolor   = maroon 
}

\usepackage{adjustbox}
\newcommand{\CommaPunct}{\mathpunct{\raisebox{0.5ex}{,}}}

\usepackage[T1]{fontenc}
\usepackage{lmodern}
\DeclareUrlCommand{\bfurl}{}

\newtheorem{thm}{Theorem}[section]

\newtheorem{lem}[thm]{Lemma}

\newtheorem{pro}[thm]{Proposition}

\theoremstyle{definition}

\newtheorem*{deftn}{Definition}
\newtheorem{exmp}{Example}[section]

\newtheorem{introthm}{Theorem}

\newtheorem{introprob}{Problem}
\newtheorem{introques}{Question}

\theoremstyle{remark}
\newtheorem{rem}{Remark}

\newcommand\intK{\smash{\mathring{\nu(K_1)}}}
\newcommand\intJ{\smash{\mathring{\nu(K_2)}}}

\DeclareMathOperator{\Hom}{Hom}
\DeclareMathOperator{\SU}{SU}

\newlength\Colsep
\title{On homology planes and contractible $4$-manifolds}

\author{Rodolfo Aguilar Aguilar}
\address{Institute of the Mathematical Sciences of the Americas, University of Miami, Coral Gables, FL 33124, USA}
\email{\href{mailto:aaguilar.rodolfo@gmail.com}{aaguilar.rodolfo@gmail.com}}
\urladdr{\url{https://sites.google.com/view/rodolfo-aguilar/}}

\author{O{\u{g}}uz \c{S}avk}
\address{CNRS and Laboratorie de Math\'ematiques Jean Leray, Nantes Universit\'e, 44322 Nantes, France}
\email{\url{oguz.savk@cnrs.fr}, \url{oguz.savk@univ-nantes.fr}}
\urladdr{\url{https://sites.google.com/view/oguzsavk/}}

\date{}

\begin{document}

\begin{abstract}

We call a non-trivial homology sphere a \emph{Kirby-Ramanujam sphere} if it bounds both a homology plane and a Mazur or Po\'enaru manifold. In 1980, Kirby found the first example by proving that the boundary of the Ramanujam surface bounds a Mazur manifold and it has remained a single example since then. By tracing their initial step, we provide the first additional examples and we present three infinite families of Kirby-Ramanujam spheres. Also, we show that one of our families of Kirby-Ramanujam spheres is diffeomorphic to the splice of two certain families of Brieskorn spheres. Since this family of Kirby-Ramanujam spheres bound contractible $4$-manifolds, they lie in the class of the trivial element in the homology cobordism group; however, both splice components are separately linearly independent in that group. 

\end{abstract}
\maketitle

\section{Introduction}

In algebraic geometry, the \emph{homology planes} are defined to be algebraic complex smooth quasi-projective surfaces with the same homology groups of the complex plane $\mathbb{C}^2$ in integer coefficients. Ramanujam provided the first example of a homology plane \cite{R71}, today this object is known as \emph{Ramanujam surface} $W(1)$ with the dual graph shown in Figure~\ref{fig:Ramanujam}. Further, he proved that $W(1)$ is a contractible $4$-manifold with a non-trivial homology sphere boundary and it is different than $\mathbb{C}^2$ up to algebraic isomorphism. Since $W(1) \times \mathbb{C}$ is diffeomorphic but not algebraically isomorphic to the complex space $\mathbb{C}^3$, the Ramanujam surface provided the first exotic algebraic structure on $\mathbb{C}^3$. It was also used by Seidel and Smith \cite{SS05} to produce the first examples of exotic symplectic structures on Euclidean spaces which are convex at infinity.

\begin{figure}[htbp]
\begin{center}
\includegraphics[width=0.35\columnwidth]{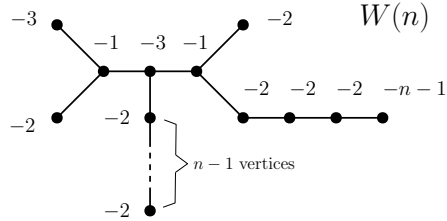}
\end{center}
\caption{A generalization of the Ramanujam surface $W(n)$.}
\label{fig:Ramanujam}
\end{figure}

The compact contractible smooth $4$-manifolds with non-trivial homology sphere boundaries were constructed in the eminent articles of Mazur \cite{M61} and Po\'enaru \cite{P60} simultaneously. The \emph{Mazur manifolds} are built with a single $0$-, $1$-, and $2$-handle, so they are obtained by adding an appropriate $2$-handle to the unknotted disk exterior of the $4$-ball $B^4$. Later, Mazur manifolds were systematically explored in the celebrated work of Akbulut and Kirby \cite{AK79}. In a similar vein, the \emph{Po\'enaru manifolds} are obtained by attaching an appropriate $2$-handle to a ribbon disk exterior of $B^4$, see Section~\ref{ldt} for details. Since contractible $4$-manifolds are core objects of Akbulut corks \cite{A91}, they have been extensively studied in low-dimensional topology, see the recent papers \cite{DHM20, HP20, Akb21, L22}. 

A classical problem in low-dimensional topology asks which homology spheres bound contractible $4$-manifolds, see \cite[Problem 4.2]{Ki78}. Around the 1980s, Kirby was able to find a Mazur manifold that has the same boundary as the Ramanujam surface up to diffeomorphism, see \cite[Pg.~56]{M80}. This valuable observation provides the initial motivation behind our definition. We aim to enrich the problem of $3$- and $4$-manifolds above by addressing the algebro-geometric objects as well.

\noindent \adjustbox{valign=t}{\begin{minipage}{0.6\linewidth}

\vskip\baselineskip

\begin{deftn}
A non-trivial homology sphere is said to be a \emph{Kirby-Ramanujam sphere} if it bounds both a homology plane and a Mazur or Po\'enaru manifold.
\end{deftn}
\end{minipage}}%
\hfill
\adjustbox{valign=t}{\begin{minipage}[t]{0.5\linewidth}
\begin{center}
\includegraphics[width=0.6\columnwidth]{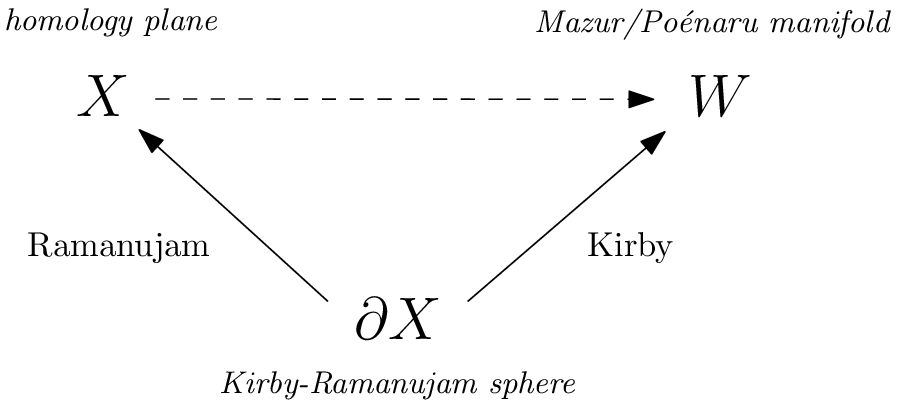}
\end{center}
\end{minipage}}

\vskip\baselineskip

Our definition also fits in complex geometry. Note that homology planes are always affine \cite{Fuj82}, hence they are all Stein \cite[Chapter~VI.3]{H70}. Further, smooth Stein $4$-manifolds have handle decompositions of index $\leq 2$ \cite{E90, G98}. Here, our choice of contractible $4$-manifolds is not random, we study such $4$-manifolds constructed directly by using ribbon knots.

After the ground-breaking work of Ramanujam, several novel techniques for the constructions of homology planes appeared in the works of Gurjar and Miyanishi \cite{GM88}, tom Dieck and Petrie \cite{DP89, DP93}, and Zaidenberg \cite{Za93}. We use their constructions for the existence of our homology spheres and we present the first new examples after Kirby and Ramanujam.

\begin{introthm} 
\label{resolution}
All dual graphs are depicted in Figure~\ref{fig:all}.
\begin{itemize}[leftmargin=2em]

\item Let $X(n)$ be the dual graphs of tom Dieck-Petrie homology planes of log-Kodaira dimension $2$. Then $\partial X(n)$ bound Mazur manifolds with one $0$-handle, one $1$-handle and one $2$-handle.

\item Let $Y(n)$ be the dual graphs of Zaidenberg homology planes of log-Kodaira dimension $2$. Then $\partial Y(n)$ bound Po\'enaru manifolds with one $0$-handle, two $1$-handles and two $2$-handles.

\item Let $Z(n)$ be the dual graphs of Gurjar-Miyanishi homology planes of log-Kodaira dimension $1$. Then $\partial Z(n)$ bound Po\'enaru manifolds with one $0$-handle, $n+1$ $1$-handles and $n+1$ $2$-handles.
\end{itemize}

Therefore, $\partial X(n)$, $\partial Y(n)$, and $\partial Z(n)$ are all Kirby-Ramanujam spheres.

\end{introthm}

\begin{figure}[htbp]
\begin{center}
\includegraphics[width=0.65\columnwidth]{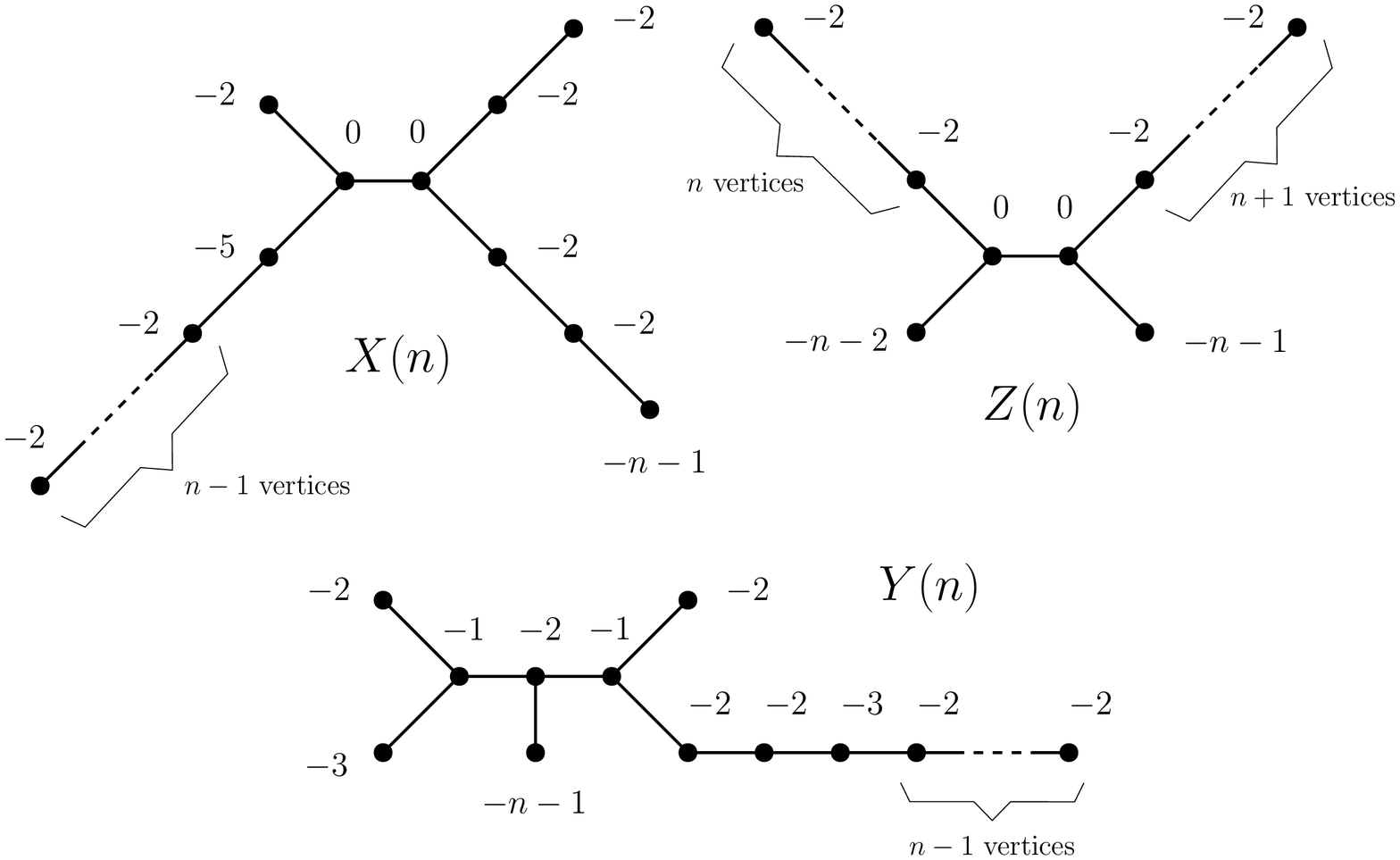}
\end{center}
\caption{The dual graphs of homology planes $X(n)$, $Y(n)$, and $Z(n)$.}
\label{fig:all}
\end{figure} 

The generalization of the Ramanujam surface shown in Figure~\ref{fig:Ramanujam} appeared in \cite{OS20}, and the second author proved that $\partial W(n)$ bound Po\'enaru manifolds with one $0$-handle, two $1$-handles, and two $2$-handles. During the course of this paper, we also investigate that $W(n)$ originate from the homology planes for every $n \geq 1$, see Proposition~\ref{savk}. Therefore, they also provide examples of Kirby-Ramanujam spheres. Since our all Kirby-Ramanujam spheres bound contractible $4$-manifolds, they are homology cobordant to the $3$-sphere $S^3$; and therefore, they represent the trivial element in the homology cobordism group $\Theta^3_\mathbb{Z}$, see survey articles \cite{M18, S22}. 

The algebraic complexity of the structure of $\Theta^3_\mathbb{Z}$ has been always studied by using \emph{Brieskorn spheres} $\Sigma(p,q,r)$, which are the links of the complex surface singularities $x^p +y^q +z^r =0$ in the complex space $\mathbb{C}^3$ where $p,q$, and $r$ are pairwise relatively prime integers. Their links are \emph{Seifert fibered spheres} over the $2$-sphere with three singular fibers having the multiplicities $p,q$, and $r$. In \cite{DHST18}, Dai, Hom, Stoffregen, and Truong provided the first family of Brieskorn spheres generating an infinite rank summand $\mathbb{Z}^\infty$ in $\Theta^3_\mathbb{Z}$. Recently, even more families were found by Karakurt and the second author \cite{KS22}.

Next, we show that one of our families of Kirby-Ramanujam spheres comes from the splice of two resolutions of Brieskorn singularities. Note that the $n=1$ case corresponds to the splice of the Poincar\'e homology sphere $\Sigma(2,3,5)$ and $\Sigma(2,3,7)$ along their singular fibers of degree $5$ and $7$.

\begin{introthm} 
\label{splice}

Let $K_1 (n)=K(n^2 +3n +1)$ and $K_2 (n)=K(n^2 +3n +3)$ denote the singular fibers of Brieskorn spheres $\Sigma_1 (n) = \Sigma(n+1, n+2, n^2 +3n +1)$ and $\Sigma_2 (n) = \Sigma(n+1, n+2, n^2 +3n +3)$ respectively. Then there is a diffeomorphism between Kirby-Ramanujam spheres $\partial Z(n)$ and the splice of $\Sigma_1 (n)$ and $\Sigma_2 (n)$ along $K_1 (n)$ and $K_2 (n)$: $$ \partial Z(n) \approx \Sigma_1 (n) \tensor[_{K_1 (n)}]{\bowtie}{_{K_2 (n)}} \Sigma_2 (n) \ \ \ \text{for each} \ n \geq 1.$$
\end{introthm}

From perspectives of both algebraic geometry and low-dimensional topology, the result presented above is to be considered an unexpected novelty due to the following reasons:

\begin{itemize}[leftmargin=2em]
\item Connections between low-dimensional topology and algebraic geometry are classically made through singularity theory, see for example N\'{e}methi's book \cite{N22}. Since their intersection matrices are not negative-definite, the dual graphs in Figure~\ref{fig:all} do not originate from the resolution of singularities in normal surfaces, see Artin's article \cite[Proposition~2]{A66}. On the other hand, there are rational homology planes that arise from complex normal surface singularities, see Wahl's paper \cite{W21}.

\item Brieskorn spheres may bound Mazur and Po\'enaru manifolds, see \cite{S20} and references therein. However, one cannot realize any Brieskorn sphere as a boundary of a homology plane due to Orevkov \cite{O97}.

\item The Brieskorn spheres $\Sigma_1 (n)$ and $\Sigma_2 (n)$ uniquely bound negative-definite resolution dual graphs, shown in Figure~\ref{fig:torussurgeries}. Using the algorithm of Neumann and Raymond \cite[Section~7]{NR78}, one can easily compute the Neumann-Siebenmann invariants of $\Sigma_2 (n)$ as follows: \[
\bar{\mu} \left (\Sigma_2 (n) \right ) = \begin{cases} 
\frac{ \frac{n+1}{2} \cdot \left ( \frac{n+1}{2} +1 \right) }{2} = \frac{n^2 + 4n +3}{8} \CommaPunct & \text{if} \ n \geq 1 \ \text{odd}, \\
\frac{ \frac{n}{2} \cdot \left ( \frac{n}{2} +1 \right) }{2} = \frac{n^2 + 2n}{8} \CommaPunct & \text{if} \ n \geq 2 \ \text{even}. 
\end{cases} \]

\noindent Since $\bar{\mu}$ is splice additive due to Saveliev \cite[Theorem~1]{Sav95}, we say that $\bar{\mu} \left (\Sigma_1 (n) \right ) = -\bar{\mu} \left (\Sigma_2 (n) \right )$ by using Theorem~\ref{resolution} and Theorem~\ref{splice}. Thus, $\Sigma_1 (n) $ and $\Sigma_2 (n)$ are homology cobordant to neither $S^3$ nor each other for each values of $n \geq 1$ because $\bar{\mu}$ is a homology cobordism invariant for plumbings, see Saveliev's article \cite{Sav02b}.

\item Furuta's gauge theoretic argument \cite{F90} guarantees that $\{ \Sigma_1 (n) \}_{n=1}^\infty$ are linearly independent in $\Theta^3_\mathbb{Z}$ since their Fintushel-Stern $R$-invariant is positive \cite{FS85, NZ85}. Relying on the computations \cite{Nem07} arising from N\'{e}methi's lattice homology \cite{AN05}, we know that $\Sigma_2 (n)$ have vanishing Ozsv\'ath-Szab\'o $d$-invariants \cite{OS03b} and they are of projective type. For Brieskorn spheres, the involutive $d$-invariants satisfy the equalities: $\bar{d} = d$ and $\underline{d} = -2\bar{\mu}$ \cite{DM19}. Since the difference between involutive $d$-invariants is arbitrarily large, we can apply the Floer theoretic linear independence argument of Dai and Manolescu \cite{DM19} to see that $\{ \Sigma_2 (n)  \}_{n=1}^\infty$ also generate a $\mathbb{Z}^\infty$ subgroup in $\Theta^3_\mathbb{Z}$.

\item One can extract monotone graded subroots \cite{DM19} from the graded roots \cite{AN05} by using the recipe in \cite[Section~6]{DM19}. N\'{e}methi's computations \cite{Nem07} yield that $\Sigma_2 (n)$ have complicated monotone graded subroots, depicted in Figure~\ref{fig:monotone}. Let $\Sigma'_{2} (n)$ and $\Sigma''_{2} (n)$ denote the subfamilies of $\Sigma_2 (n)$ with respect to the odd and even values of $n$, respectively. Even though their gradings are different, they have very similar monotone graded subroots as the family of Dai, Hom, Stoffregen, and Truong. By mimicking their calculations, one can compute their invariants to conclude that $\{ \Sigma'_2 (n)  \}_{n=1}^\infty$ and $\{ \Sigma''_2 (n) \}_{n=2}^\infty$ both generate $\mathbb{Z}^\infty$ summands in $\Theta^3_\mathbb{Z}$. However, we are not able to distinguish the homology cobordism classes of $\Sigma'_{2} (n)$ and $\Sigma''_{2} (n)$.  For our first family $\Sigma_1 (n)$, the corresponding monotone graded subroots are trivial due to Tweedy's calculations \cite{T13} since $\underline{d} \left (\Sigma_1 (n) \right ) = \bar{d} \left (\Sigma_1 (n) \right )$. Therefore, we cannot address the invariants of Dai, Hom, Stoffregen, and Truong.
\end{itemize}

\begin{figure}[htbp]
\begin{center}
\includegraphics[width=0.55\columnwidth]{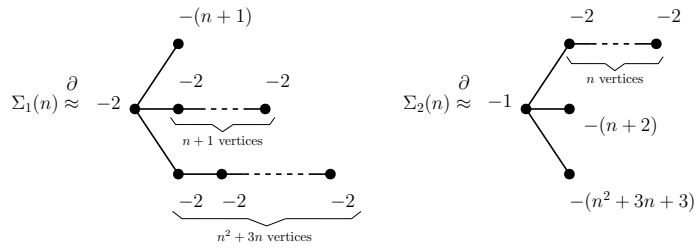}
\end{center}
\caption{The resolution dual graphs for $\Sigma_1 (n)$ and $\Sigma_2 (n)$.}
\label{fig:torussurgeries}
\end{figure}

\begin{figure}[htbp]
\begin{center}
\includegraphics[width=0.55\columnwidth]{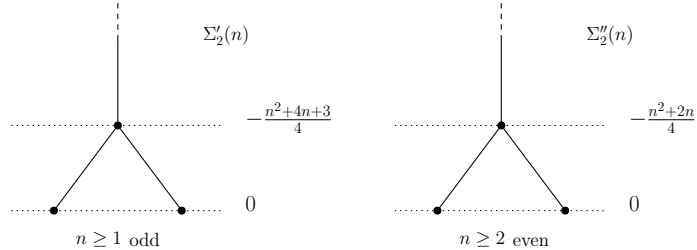}
\end{center}
\caption{The monotone graded subroots of $\Sigma_2 (n)$.}
\label{fig:monotone}
\end{figure}

In \cite[Section~6]{Sav98}, Saveliev computed instanton Floer homology of surgeries along the generalized square knots. His result includes our family $ \Sigma_1 (n) \tensor[_{K_1 (n)}]{\bowtie}{_{K_2 (n)}} \Sigma_2 (n) $ as well. He also described the representation spaces of irreducible representations of the fundamental groups of splices in $\SU (2)$ in \cite[Section~5]{Sav98}. Our other Kirby-Ramanujam spheres a priori have the splice decompositions of some Brieskorn spheres, but we do not know currently what they are. After finding these components, one can study these objects by following Saveliev's strategies or decipher their homology cobordism classes by using our approaches presented in the paper. 

\subsection*{Organization} In Section \ref{ag}, we review the constructions of tom Dieck-Petri, Zaidenberg, and Guryar-Miyanishi types of homology planes. In Section \ref{ldt}, we provide a background for the Mazur and Po\'enaru manifolds, and the splicing of homology spheres. We present the proofs of Theorem \ref{resolution} and Theorem \ref{splice} together in Section \ref{proofs}. Finally, we discuss some further research topics and list several open problems in Section~\ref{further}.

\subsection*{Acknowledgements}
We are indebted to Alexander I. Suciu for his initiative in bringing us as collaborators. We are very grateful to Marco Golla, Nikolai Saveliev, and Mikhail G. Zaidenberg for sharing their expertise with us and providing insightful suggestions. We would also like to thank Louis Funar and Andras N\'emethi for their valuable comments. Finally, we would like to thank the anonymous referee for their careful assessment and invaluable feedback.

This project started when O\c{S} visited the Max-Planck-Instit\"ut f\"ur Mathematik in Bonn, we want to thank them for their hospitality and support. RAA was supported by the Ministry of Education and Science of the Republic of Bulgaria through the Scientific Program \say{Enhancing the Research Capacity in Mathematical Sciences (PIKOM)}. O\c{S} was supported by the Turkish Fulbright Commission \say{Ph.D. dissertation research grant}. O\c{S} thanks his advisor \c{C}a\u{g}r{\i} Karakurt and his mentor Ciprian Manolescu for their constant guidance and support during the course of writing this article.

\section{Preliminaries in Algebraic Geometry}
\label{ag} 

Consider a curve $D=D_1\cup\ldots\cup D_i$ in a projective smooth surface $\bar{Y}$ with the irreducible components $D_j$ of $D$. We say that the curve $D$ is \emph{simple normal crossing} or \emph{SNC} for short if each $D_j$ is smooth and around every point $p\in D$ there exists local complex coordinates $(z_1,z_2)$ such that $D=\{z_1z_2=0\}$ or $D=\{z_1=0\}$. For any SNC curve $D$, we can associate a \emph{dual graph} $\Delta_D$ as follows:

\begin{itemize}[leftmargin=2em]
\item For every irreducible component $D_j$,  we have a vertex $v_j$,
\item For every point in $D_k\cap D_l$, we have an edge connecting $v_k$ and $v_l$,
\item Every curve $D_j$ has a self-intersection number $D_j^2$ corresponding to weight of the vertex $v_j$.
\end{itemize}

The weights constitute the diagonal entries of the intersection matrix of $\Delta_D$. For the rest, we write $1$ if there is an edge connecting vertices and we assign $0$ otherwise. 

An algebraic complex smooth quasi-projective surface $X$ is called a \emph{homology plane} if it has the same of homology groups of $\mathbb{C}^2$ in integer coefficients, i.e., $H_* (X; \mathbb{Z}) = H_* (\mathbb{C}^2; \mathbb{Z})$. By definition of quasi-projectivity, we have that $X$ is a locally closed subset of the complex projective space $\mathbb{CP}^n$. By taking its closure inside $\mathbb{CP}^n$ we find a projective variety $\bar{X}$. By blowing up points, this compactification $\bar{X}$ can be chosen to be a smooth projective surface such that $D=\bar{X}\setminus X$ is an SNC curve. 

Moreover, for a homology plane $X$, there always exists a compactification $\bar{X}$ such that the dual graph $\Delta_D$ for $D=\bar{X}\setminus X$ is \emph{absolutely minimal}, in the sense that the weight of any of at most linear vertex (i.e. linear or end) of $\Delta_D$ does not exceed $-2$. See \cite[Appendix~A.2]{Za93} for details. Such a dual graph determines uniquely the homology plane $X$, consult \cite[Remark~A.2.5]{Za93}. Therefore, we will use $X$ to denote both the homology plane and its absolutely minimal dual graph. 

The SNC curve $D$ is of a very special type due to the following folklore result, for a proof see \cite[Proposition 2.1]{DP93}.

\begin{pro}
Let $X$ be a homology plane and let $D=D_1\cup \ldots \cup D_i=\bar{X}\setminus X$ be an SNC curve. Then the irreducible components $D_j$ of $D$ are isomorphic to $\mathbb{CP}^1$. Moreover, the dual graph $\Delta_D$ is a tree, i.e., it is connected and has no cycles.
\end{pro}

Therefore, there exists a compact regular tubular neighborhood $U$ of $D$ in $\bar{X}$ and its boundary $\partial U$ is a \emph{plumbed homology sphere} with the plumbing graph $\Delta_D$. By abusing of language, we also call $\partial U$ \emph{the boundary of the homology plane} $X$. 

\subsection{The Classification of Homology Planes}

The partial classification of homology planes follows the lines of classification of open smooth algebraic surfaces. We refer to the book of Miyanishi \cite{MM01} for an overview of this classification. A main ingredient is the invariant called the logarithmic Kodaira dimension.

Let $X$ be a homology plane and $\bar{X}$ a smooth projective compactification with an SNC $D=\bar{X}\setminus X$. Define its \emph{logarithmic Kodaira dimension} $\bar{k}(X)\in \{-\infty,0,1,2\}$ as the unique value that satisfies the following inequalities: for some positive constants $\alpha$ and $\beta$ we have
$$\alpha m^{\bar{k}(X)}\leq \dim H^0(\bar{X},m(K_{\bar{X}}+D) )\leq \beta m^{\bar{k}(X)}$$ where $m$ is sufficiently large and divisible positive integer. Here, $H^0$ stands for the sheaf cohomology and $K_{\bar{X}}$ denotes the canonical divisor of $\bar{X}$. The invariant $\bar{k}(X)$ does not depend on the chosen compactification $\bar{X}$, see \cite[Chapter~11.1]{I77}.

We know that $\bar{k}(X)=-\infty$ if and only if $X\cong \mathbb{C}^2$ as algebraic varieties, see the articles of Fujita \cite{Fuj79} and Miyanishi and Sugie \cite{MS80}. Fujita also prove that if $X$ is not isomorphic to $\mathbb{C}^2$ then $\bar{k}(X) \geq 1$ \cite{Fuj82}. Furthermore, the homology planes with $\bar{k}(X) = 1$ are completely classified by Gurjar and Miyanishi in \cite{GM88}. Also, there are explicit constructions of homology planes with $\bar{k}(X) = 2$ due to tom Dieck and Petrie \cite{DP93} and Zaidenberg \cite{Za93, Za94}.

In order to construct infinite families of homology planes, we first define an operation on a dual graph $\Delta_D$ of an SNC curve $D$ in a projective surface $\bar{Y}$ called \emph{expanding an edge}.

Let $v_k$ and $v_l$ be two vertices of $\Delta_D$ and let $e$ be an edge joining them. Fix two coprime positive integers $a$ and $b$. We use the following short-hand notation for the continued fraction expansion: \begin{equation*}
[d_s, d_{s-1} \ldots,d_1]=d_s-\cfrac{1}{d_{s-1}-\cfrac{1}{\cdots-\cfrac{1}{d_1}}} \ \cdot
\end{equation*}

Now set $$ \frac{a+b}{b}=[c_{-r},\ldots,c_{-1}] \ \ \ \text{and} \ \ \ \frac{a+b}{a}=[c_s,\ldots,c_1]. $$ Replace the edge $e$ by a linear graph divided in three linear subgraphs with $r-1, 1$ and $s-1$ vertices respectively: $$\Delta_{a,b}=\{E_{-r+1},\ldots,E_{-1}\},\quad  E_0, \quad \Delta_{a,b}'=\{E_1,\ldots, E_{s-1}\}$$ with weights:
$$E_h^2=\left\lbrace \begin{array}{rl}
-c_h , & \text{if } h\not =0 ,\\
-1 , & \text{if }  h=0 ,
\end{array} \right.$$
and we modify the weights of $v_k$ and $v_l$ to $D_k^2-c_{-r}+1$ and $D_l^2-c_s+1$ respectively, see Figure~\ref{fig:Expanding}.

\begin{figure}[htbp]
\begin{center}
\includegraphics[scale=1]{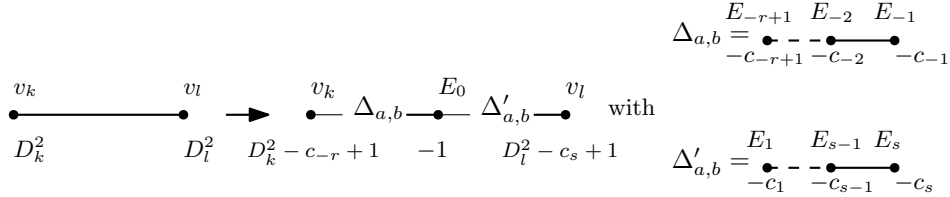}
\end{center}
\caption{Expanding an edge.}
\label{fig:Expanding}
\end{figure}

\begin{lem}
\label{lem:Partial1} 
Consider $a=n$ and $b=1$, then we have $$\frac{a+b}{b}=n+1 \quad
\frac{a+b}{a}=[\underbrace{2,2,\dots,2}_{n\text{-times}}] \cdot$$
\end{lem}

\subsection{Gurjar-Miyanishi Homology Planes}
\label{gmhom}

Now we review the following constructions of homology planes appeared in \cite{GM88}. We call the resulting surfaces \emph{Gurjar-Miyanishi homology planes}. They can be obtained in the following fashion.

Suppose that $\Delta_1\subset \Delta_D$ is a cycle, that the vertices $v_k,v_l$ and the edge $e$ are in $\Delta_1$. We call the following procedure \emph{cutting a cycle}: 

\begin{itemize}[leftmargin=2em]
\item Fix two coprime positive integers $a$ and $b$ and expand the edge $e$ as above,
\item Remove the vertex $E_0$ and its two adjacent edges.
\end{itemize}

Consider the union of four lines in the complex projective plane $$L(1,4)  \doteq \cup_{k=1}^4 l_k\subset \mathbb{CP}^2$$ where the first three lines intersect in a point $P$ and the fourth line is in general position. Blow up the point $P$ to obtain an SNC curve $D$. Denote its dual graph by $\Delta_D$. We will cut the cycles in $\Delta_D$ at the edges corresponding to the points $l_1\cap l_4$ and $l_2\cap l_4$ and two pairs of coprime positive integers $a,b$ and $c,d$ such that $ac-ad-bc=\pm 1$. This step is expressed by dashed lines in Figure \ref{fig:GM}. 

\begin{figure}[htbp]
\begin{center}
\includegraphics[scale=.7]{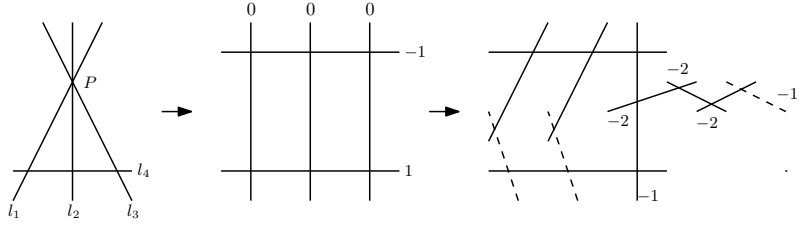}
\end{center}
\caption{The procedure for the construction of Gurjar-Miyanishi homology planes.}
\label{fig:GM}
\end{figure}

We also need to blow up a smooth point in the strict transform of $l_4$, which can be followed by a sequence of $k$ many blow up operations, and each time a smooth point of the exceptional divisor is created by the precedent blow up. This yields a chain of $-2$-curves and one $-1$-curve. We do not consider this $-1$-curve in the divisor $D'$, whose dual graph is shown in Figure~\ref{fig:GM-dg}. The divisor $D'$ lies in a projective smooth surface $\bar{X}$ obtained by sequences of blow ups in $\mathbb{CP}^2$.

\begin{figure}[htbp]
\begin{center}
\includegraphics[scale=0.8]{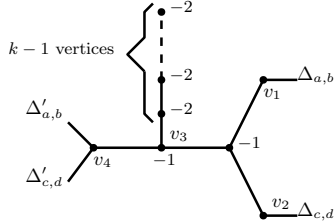}
\end{center}
\caption{The dual graph of Gurjar-Miyanishi homology planes.}
\label{fig:GM-dg}
\end{figure}

The Gurjar-Miyanishi homology plane corresponding to the numerical data $(k,(a,b),(c,d))$ is given by $X \doteq \bar{X}\setminus D'$. If $(a,b)\not =(1,1)\not = (c,d)$, then $X$ has logarithmic Kodaira dimension $1$ and it is contractible. 

\begin{lem}\label{lem:CalcFraction} Let $(a,b)=(n+1,n)$. Then $$\frac{a+b}{a}=[2,n+1], \quad \frac{a+b}{b} = [3,\underbrace{2,2,\dots,2}_{n\text{-times}}].$$
\end{lem} 

\subsection{tom Dieck-Petri Homology Planes}
\label{tdphom}

Next, we consider homology planes introduced by tom Dieck and Petri. In Figure~\ref{fig:tDP-DP}, we have the line arrangement $L(3)$ together with its blow ups at two points of multiplicity three. 

In this way, we can obtain an SNC curve and a dual graph $\Delta_D$ where we cut the cycles corresponding to intersections of lines $L_2\cap L_3, L_2\cap L_7, L_5\cap L_6$, and $L_1\cap L_4$. Here, the cutting data $(1,1)$ stands for the first three intersections and $(a,b)$ for the last one. The resulting dual graphs depicted in Figure \ref{fig:tDP} are called \emph{tom Dieck-Petri homology planes}.

\begin{figure}[htbp]
\begin{center}
\includegraphics[scale=.7]{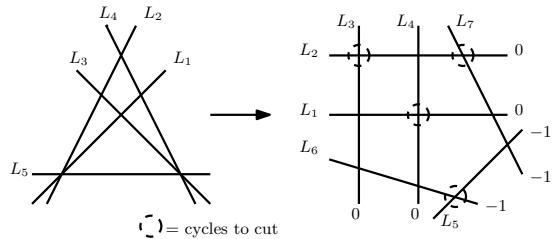}
\end{center}
\caption{The $L(3)$ arrangement.}
\label{fig:tDP-DP}
\end{figure}

\begin{figure}[htbp]
\begin{center}
\includegraphics[scale=.7]{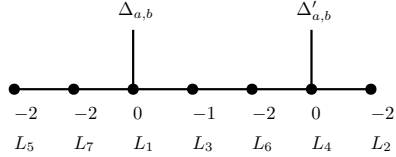}
\end{center}
\caption{The dual graph of tom Dieck-Petri homology planes.}
\label{fig:tDP}
\end{figure}

\begin{rem}
In order to obtain a homology plane from the conditions $(a,b)$ above, they must satisfy: $4b-3a=\pm 1$, see \cite[Section~5.2.3]{AA21}.
\end{rem}

\begin{lem}
\label{tdpfrac}
For $a=4n+1$ and $b=3n+1$, we have 
$$\frac{a+b}{b}=[3,2,2,n+1], \quad \frac{a+b}{a}=[2,5, \underbrace{2,2,\dots,2}_{n-1 \text{-times}}].$$
\end{lem}

\subsection{Zaidenberg Homology Planes}
\label{zhom}

We finally consider the homology planes constructed by Zaidenberg \cite{Za93}. His approach provides the generalization of the original Ramanujam surface \cite{R71}. 

Consider $\mathbb{CP}^1\times\mathbb{CP}^1$ with coordinates $(x:y,u:v)$ and the curves 

\begin{align*}
c_s=\{u^2 y^s  = x^s v^2\}, \quad e_0=\{v=0\}, \quad e_1=\{y=0\}, \quad l_1=\{x=y\}.
\end{align*} 

\noindent Denote by $d_s=c_s\cup e_0\cup e_1\cup l_1$, let $\pi:\bar{X}'_s\to \mathbb{CP}^1\times \mathbb{CP}^1$ be the minimal resolution of singularities of $d_s$ and let $B_s=\pi^{-1}(d_s)$.  In Figure~\ref{fig:Zai}, there is a representation in the affine plane $\mathbb{C}^2$ with coordinates $(x,u)$ of the curve $d_s$.

\begin{figure}[htbp]
\begin{center}
\includegraphics[scale=.5]{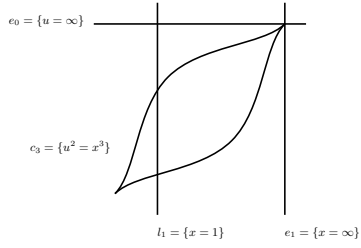}
\end{center}
\caption{Zaidenberg's curves.}
\label{fig:Zai}
\end{figure}

Consider the integer valued $2 \times 2$ matrix $$M=\left(\begin{array}{ll}
a &b \\
c &d
\end{array} \right), \ \ \ \mathrm{det}(M) = 1$$
and apply the cutting cycles procedure to $(\bar{X}'_s,B_s)$ in two points $z_1$ and $z_2$ corresponding to $l_1\cap c_s$ according to the numerical data in $M$ and denote by $(\bar{X}_s,W_s)$ the resulting pair. 
Zaidenberg proved that the surface $X_s=\bar{X}_s\setminus W_s$ is a contractible homology plane \cite{Za93}.

Here we only use the case of $s=3$. A partial resolution of singularities of the curve $d_s$ is shown in Figure~\ref{fig:ZaiRes}. This can be used to obtain the dual graph of $W_s$, depicted in Figure \ref{fig:ZaiDualGraph}.  
They are said to be \emph{Zaidenberg homology planes}.

\begin{figure}[htbp]
\begin{center}
\includegraphics[scale=1]{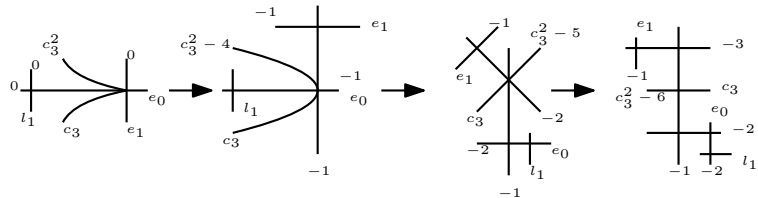}
\end{center}
\caption{A partial resolution of singularities of the curve $d_s$.}
\label{fig:ZaiRes}
\end{figure}

\begin{figure}[htbp]
\begin{center}
\includegraphics[scale=.75]{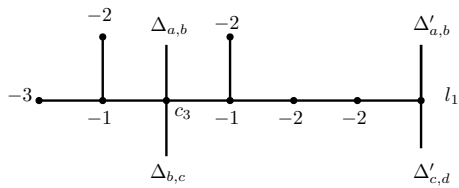}
\end{center}
\caption{The dual graph of the divisor $W_s$.}
\label{fig:ZaiDualGraph}
\end{figure}

\begin{exmp}
The original Ramanujam surface corresponds to the following numerical data: \begin{center}
$s=3 \ \ \text{and} \ \ M=\left(\begin{array}{ll}
1 &1 \\
1 &2
\end{array} \right) \cdot $ 
\end{center}
\end{exmp}

Analysing the Zaidenberg's construction, we see that Ramanujam manifolds studied by the second author in \cite{OS20} are also homology planes. 

\begin{pro}
\label{savk}
Let $W(n)$ be Ramanujam manifolds depicted in Figure~\ref{fig:Ramanujam}. Then the boundaries $\partial W(n)$ are Kirby-Ramanujam homology spheres.
\end{pro}
\begin{proof}

By \cite{OS20}, we already know that the homology spheres $\partial W(n)$ bound Po\'enaru manifolds. Also, their dual graph corresponds to the numerical data $s=3$ and $M=\left(\begin{array}{ll}
1 &1 \\
n &n+1
\end{array} \right)$.
Using Lemma \ref{lem:CalcFraction} and the dual graph in Figure \ref{fig:ZaiRes}, we are done.
\end{proof}

\section{Preliminaries in Low-Dimensional Topology}
\label{ldt}

\subsection{Mazur and Po\'enaru Manifolds}

A knot $K$ in $S^3$ is said to be a \emph{ribbon knot} if it can be built by attaching bands between components of an unlink. The minimal number of bands is recorded by the quantity called \emph{fusion number}. Since the Euler characteristic of a disk is one, the fusion number also determines the number of components of the corresponding unlink. A knot $K$ in $S^3$  is called a \emph{slice knot} if it bounds an embedded smooth disk in $B^4$. In other words, ribbon knots are slice knots whose slice disks have no $2$-dimensional $2$-handles. The disks corresponding to ribbon knots are said to be \emph{ribbon disks}.

The basic examples of \emph{ribbon knots} are the unknot $U$, the square knot $T(2,3) \# \overline{T(2,3)}$, and the generalized square knot $T(n+1,n+2) \# \overline{T(n+1,n+2)}$ for $n \geq 1$. They respectively have fusion numbers $0$, $1$, and $n$, see \cite[Section~1.7]{JMZ20}. Here, the notation $T(p,q)$ stands for the left-handed $(p,q)$-torus knot given coprime positive integers $p$ and $q$ and $\overline{T(p,q)}$ denotes its mirror image.

The Po\'enaru manifolds are defined by attaching a $4$-dimensional $2$-handle to ribbon disk exteriors of $B^4$ in order to kill their fundamental groups. This can be achieved easily by using a knot in the $3$-manifold obtained by $0$-surgery on the ribbon knot, provided that the former knot normally generates the fundamental group of ribbon disk exteriors of $B^4$.  The original construction of Po\'enaru addresses the square knot \cite{P60} and the strategy can be generalized to all ribbon knots directly, see \cite[Pg.289-290]{AG96}. The slice disks may have $2$-dimensional $2$-handles; therefore, the handle decomposition of the corresponding Po\'enaru manifolds may include $3$-handles. Since we want to control the number of handles of the $4$-manifold precisely, we prefer to work with ribbon knots. 

\begin{lem}[Lemma~2.1, \cite{OS20}]
\label{mazur}
Let $Y$ be the $3$-manifold obtained by $0$-surgery on a ribbon knot in $S^3$ with the fusion number $n \geq 1$. Then any homology sphere obtained by an integral surgery on a knot in $Y$ bounds a Po\'enaru manifold with one $0$-handle, $n+1$ $1$-handles, and $n+1$ $2$-handles. Further, if the initial ribbon knot is the unknot, then the resulting homology sphere bounds a Mazur manifold with one $0$-handle, one $1$-handle, and one $2$-handle.
\end{lem}

In order to shorten the proof of Theorem~\ref{resolution}, we need the following Kirby calculus trick of Akbulut and Larson. It was found in \cite{AL18} and played a key role in the proof of their main theorem. Later, it was also effectively used by the second author \cite{OS20, S20}. 

The \emph{Akbulut-Larson trick} is an observation about describing the iterative procedure for passing from the surgery diagram of a homology sphere to a consecutive one by using a single blow up with an isotopy, see Figure~\ref{fig:trick}.

\begin{figure}[htbp]
\begin{center}
\includegraphics[width=0.5\columnwidth]{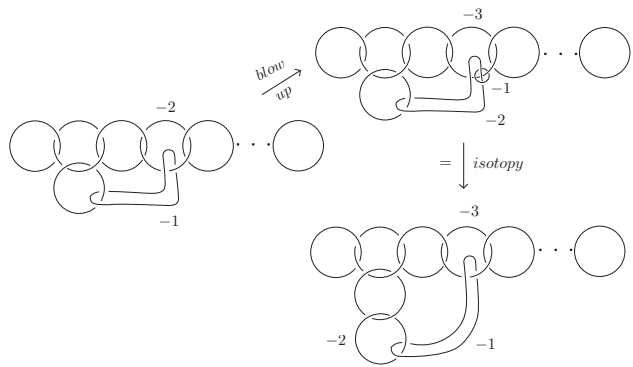}
\end{center}
\caption{The Akbulut-Larson trick.}
\label{fig:trick} 
\end{figure}

\subsection{Splicing Homology Spheres}
\label{splicing}

The splice operation was defined by Siebenmann \cite{S80} and later it was elaborated systematically in the novel book of Eisenbud and Neumann \cite{EN85}.

Let $(Y_1, K_1)$ and $(Y_2, K_2)$ be two pairs of homology spheres and knots together with meridians and longitudes $(\mu_1, \lambda_1)$ and $(\mu_2, \lambda_2)$, and tubular neighboords $\nu(K_1)$ and $\nu(K_2)$ inside $Y_1$ and $Y_2$ . The \emph{splice} operation between $(Y_1, K_1)$ and $(Y_2, K_2)$ produces a homology sphere given by

\begin{center}
$Y_1 \tensor[_{K_1}]{\bowtie}{_{K_2}} Y_2 \doteq \left ( Y_1 \setminus \intK \right ) \cup \left ( Y_2 \setminus \intJ \right )$
\end{center}

\noindent where the pasting homeomorphism along tori boundaries sends $\mu_1$ onto $\lambda_2$ and $\lambda_1$ onto $\mu_2$. 

Since Brieskorn spheres bound unique negative definite plumbing graphs \cite[Example~1.17]{Sav02}, we can describe their splice on a joint plumbing graph by using the recipe in \cite[Chapter~V.20]{EN85}. These plumbings graphs of Brieskorn spheres are the same as their resolution dual graphs of singularities up to blow up and down, see also Neumann's plumbing calculus \cite[Theorem~18.3-18.4]{EN85}.

Let $\left (\Sigma(a_1, a_2 , a_3), K(a_n) \right )$ and $\left (\Sigma(b_1, b_2, b_3), K(b_m) \right )$ be pairs of Brieskorn spheres and singular fibers with plumbing graphs $G$ and $G'$ respectively. In Figure~\ref{fig:splice}, the arrows indicate the singular fibers. The weights $e_n$ and $e_m$ label the endmost vertices in the branches of $K(a_n)$ and $K(a_m)$ respectively. The additional weights $x$ and $y$ appeared after the splice operation are given by the formula $x = \mathrm{det}(G_0)$ and $y = \mathrm{det}(G'_0)$ where $G_0$ and $G'_0$ are the portions of $G$ and $G'$ obtained by removing $e_n$ and $e_m$ respectively. These type of determinants can be easily found, consult \cite[Chapter~V.21]{EN85}.

\begin{figure}[htbp]
\begin{center}
\includegraphics[width=0.65\columnwidth]{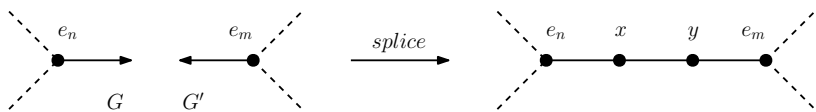}
\end{center}
\caption{The splice of plumbing graphs.}
\label{fig:splice} 
\end{figure}

\section{Proofs of Theorem~\ref{resolution} and Theorem~\ref{splice}}\label{proofs}

Since we need a portion of the proof Theorem~\ref{resolution} in the proof argument of Theorem~\ref{splice}, we first show Theorem~\ref{resolution}.

\begin{proof}[Proof of Theorem~\ref{resolution}]

We present our proof case by case.

\vspace{0.5 em}

\noindent $\bullet$ \textbf{The Family} $X(n):$ First, we verify that $X(n)$ for every $n \geq 1$ originate from tom Dieck-Petri homology planes by addressing Subsection~\ref{tdphom}. Note that we can contract $-1$'s in the dual graph shown in Figure~\ref{fig:tDP}. Then we set $a=4n+1$ and $b=3n+1$ as in  Lemma~\ref{tdpfrac}. Writing $\Delta_{a,b}$ and $\Delta_{a,b}'$ explicitly, we obtain the dual graph $X(n)$ displayed in Figure~\ref{fig:all}. Since $X(n)$ does not come from a Gurjar-Mayanishi type dual graph and it is absolutely minimal, the logarithmic Kodaira dimension of $X(n)$ is $2$, see \cite[Appendix~A.5]{Za93}

In Figure~\ref{fig:x}, we first draw the surgery diagram of $\partial X(1)$. The additional $2$-handle is shown in dark black. Applying the several Kirby moves indicated in the picture, we end up with the surgery diagram of $S^3_0(U)$. Now we can use Lemma~\ref{mazur} to conclude that $\partial X(1)$ bounds a Mazur manifold with one $0$-handle, one $1$-handle, and one $2$-handle. We can guarantee that all $\partial X(n)$ for $n \geq 2$ bound such Mazur manifolds as well by using the Akbulut-Larson trick consecutively, see Figure~\ref{fig:trick2}.

\begin{figure}[htbp]
\begin{center}
\includegraphics[width=0.65\columnwidth]{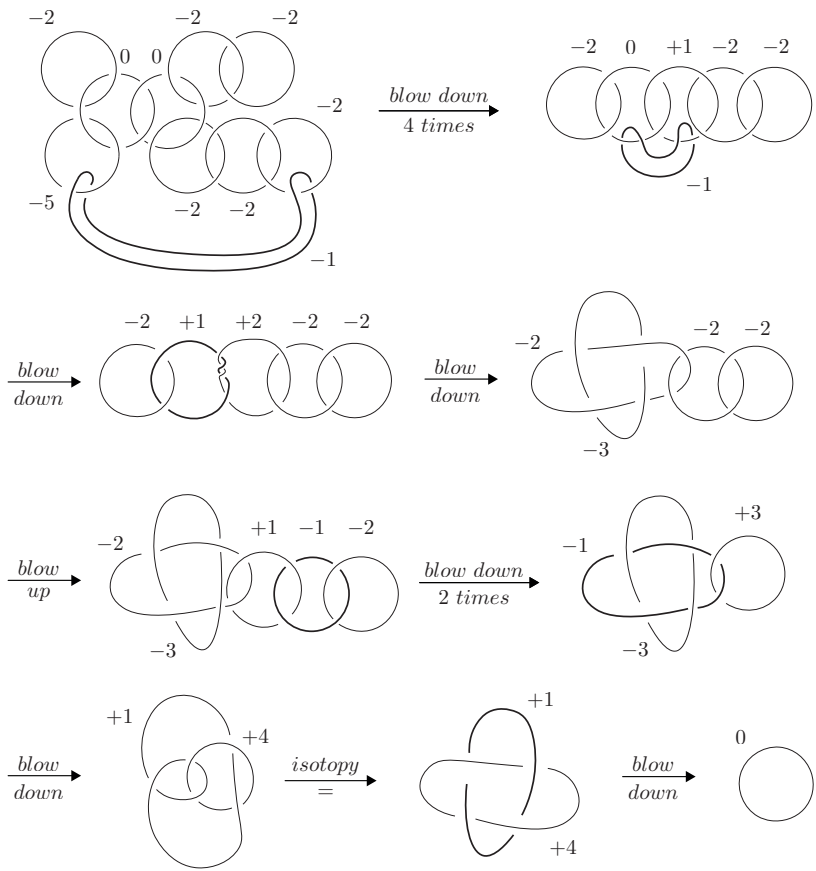}
\end{center}
\caption{$(-1)$-surgery from $\partial X(1)$ to $S^1 \times S^2 = S^3_0 (U)$.}
\label{fig:x} 
\end{figure}

\begin{figure}[htbp]
\begin{center}
\includegraphics[width=0.55\columnwidth]{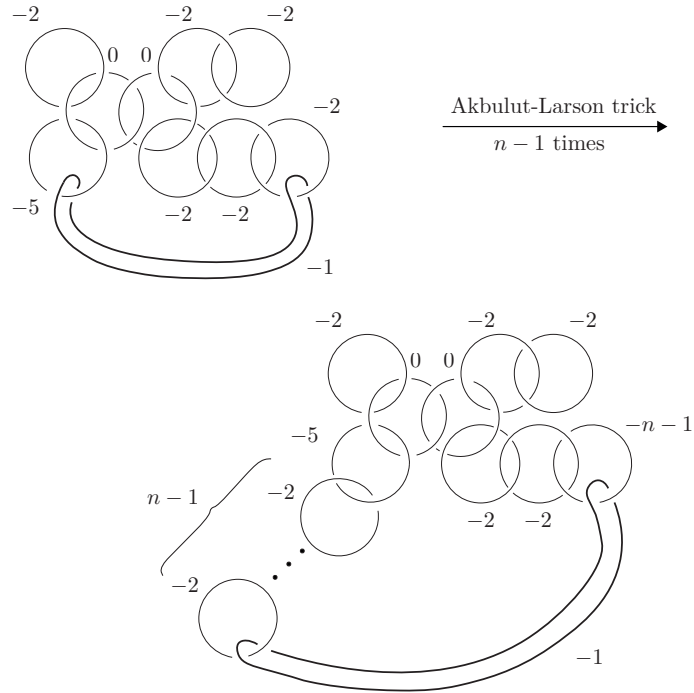}
\end{center}
\caption{The Akbulut-Larson trick for the family $X(n)$.}
\label{fig:trick2} 
\end{figure}

\vspace{0.5 em}

\noindent $\bullet$ \textbf{The Family} $Y(n):$ By following Subsection~\ref{zhom}, we initially verify that $Y(n)$ for $n \geq 1$ are indeed Zaidenberg homology planes. We use Lemma \ref{lem:CalcFraction} and the dual graph in Figure \ref{fig:ZaiRes} to guarantee that the resolution graph of $Y(n)$ is encoded by the numerical data $s=3$ and $M=\left(\begin{array}{ll}
1 &1 \\
n+1 &n
\end{array} \right)$.
Since Zaidenberg classified his homology planes in terms of their logarithmic Kodaire dimensions \cite{Za93}, we know that $\bar{k} (Y(n)) = 2$.

Then we address our previous strategy for the family $Y(n)$. We first start with the surgery diagram of $\partial Y(1)$ and then attach a $(-1)$-framed dark black $2$-handle. Now we perform various blow down operations and we eventually find the surgery diagram of $S^3_{0} (T(2,3) \# \overline{T(2,3)})$. Recall that the square knot has fusion number one, by using Lemma~\ref{mazur}, we can say that $\partial Y(1)$ bounds a Po\'enaru manifold with one $0$-handle, two $1$-handles, and two $2$-handles. Applying the Akbulut-Larson trick again, we can pass the surgery diagram of $\partial Y(2)$, and so on; therefore, we are done.

\begin{figure}[htbp]
\begin{center}
\includegraphics[width=0.65\columnwidth]{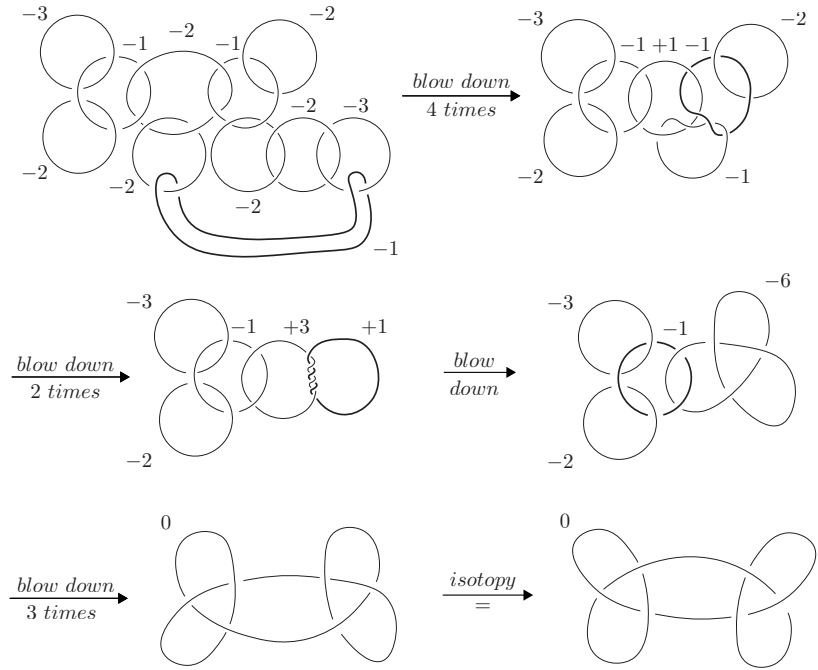}
\end{center}
\caption{$(-1)$-surgery from $\partial Y(1)$ to $S^3_{0} (T(2,3) \# \overline{T(2,3)})$.}
\label{fig:y} 
\end{figure}

\vspace{0.5 em}

\noindent $\bullet$ \textbf{The Family} $Z(n):$ We first see that $Z(n)$ for $n \geq 1$ are Gurjar-Miyanishi homology planes by following Subsection~\ref{gmhom}.  Let $k=1, (a,b)=(n+2,1)$ and $(c,d)=(n+1,n)$. Note that the vertex $v_3$ in Figure~\ref{fig:GM-dg} has weight $-1$, so we can contract it. Using Lemma~\ref{lem:Partial1} and Lemma~\ref{lem:CalcFraction}, we can write the linear graphs obtained by cutting the cycles explicitly.  Thus, the dual graph of $Z(n)$ corresponds to one shown in Figure~\ref{fig:all}. Moreover, $\bar{k} (Z(n)) = 1$.

Now, we begin with the surgery diagram of $\partial Z(n)$ for every $n \geq 1$, see Figure~\ref{fig:z}. Next, we apply a blow up and find the second picture. Performing many blow down operations along both right- and left-hand sides of our figure, we finally reach the surgery diagram of the manifold $S^3_{-1} (T(n+1,n+2) \# \overline{T(n+1,n+2)})$, compare with \cite[Proposition~1.A-1.B]{K78}. Note that the last two blow downs in Figure~\ref{fig:z} respectively change the framings as follows: 
\begin{itemize}
\item $(-1)$-blow down effect: $2n+2 + (n+1)^2 = n^2 + 4n + 3$, 
\item $(+1)$-blow down effect: $n^2 + 4n +3 - (n+2)^2 = -1$.
\end{itemize}

To address Lemma~\ref{mazur}, we can add a $(-1)$-framed  $2$-handle to pass to the surgery diagram $S^3_{0} (T(n+1,n+2) \# \overline{T(n+1,n+2)})$. Since the generalized square knot has fusion number $n$, we can conclude that for each $n \geq 1$, our Kirby-Ramanujam spheres $\partial Z(n)$ bounds a Po\'enaru manifold built with one $0$-handle, $n+1$ $1$-handles, and $n+1$ $2$-handles.

\begin{figure}[htbp] 
\begin{center}
\includegraphics[width=0.65\columnwidth]{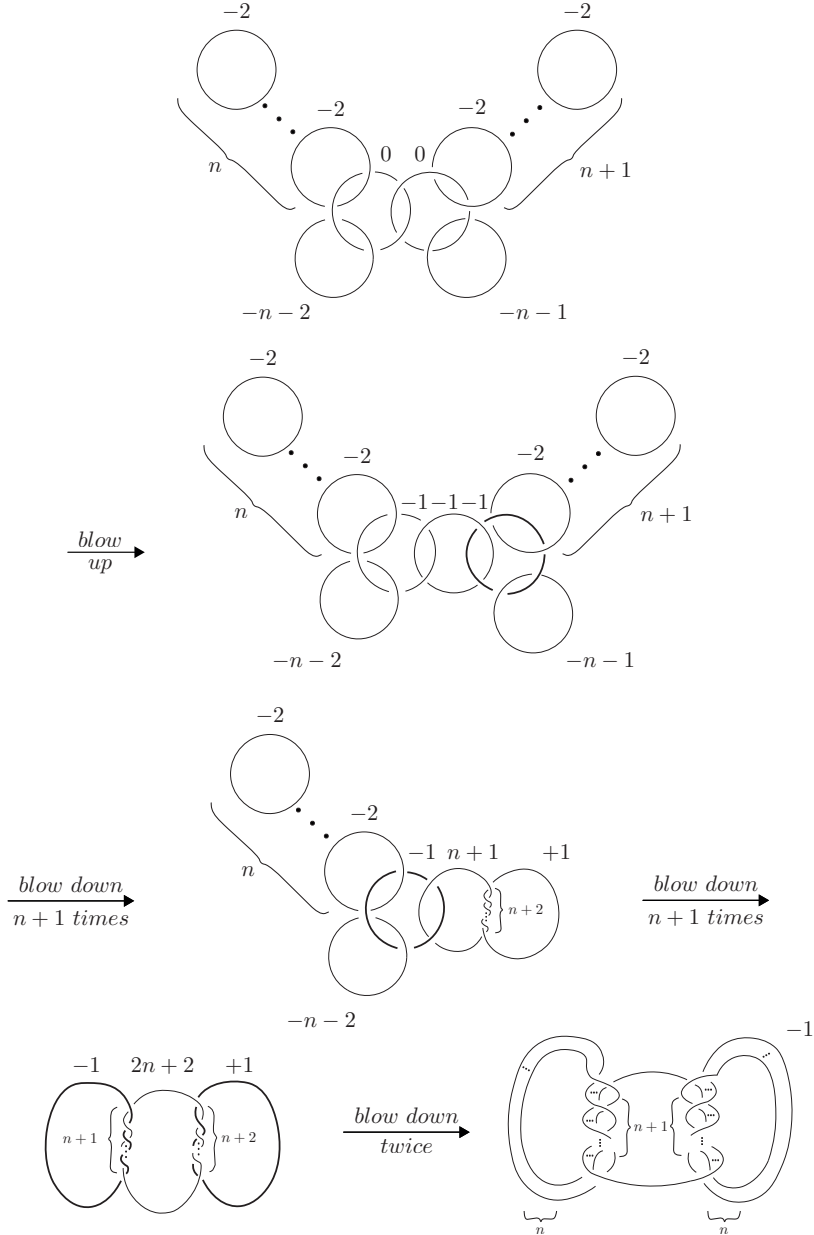}
\end{center}
\caption{A diffeomorphism between $\partial Z(n)$ to $S^3_{-1} (T(n+1,n+2) \# \overline{T(n+1,n+2)})$.}
\label{fig:z}
\end{figure}

\end{proof}

Now we are ready to prove our other theorem.

\begin{proof}[Proof of Theorem~\ref{splice}]

We already know from the previous proof that homology spheres $\partial Z(n)$ and $S^3_{-1} (T(n+1,n+2) \# \overline{T(n+1,n+2)})$ are diffeomorphic for every $n \geq 1$. We claim that there is also a diffeomorphism between the latter family of $3$-manifolds and the splice of Brieskorn spheres $\Sigma(n+1, n+2, n^2 +3n +1) \tensor[_{K(n^2 +3n +1)}]{\bowtie}{_{K(n^2 +3n +3)}} \Sigma(n+1, n+2, n^2 +3n +3).$

We first consider the base case $n=1$. Using the recipe in Subsection~\ref{splicing}, we can easily find the plumbing graph of $\Sigma(2,3,5) \tensor[_{K(5)}]{\bowtie}{_{K(7)}} \Sigma(2,3,7)$, shown in Figure~\ref{fig:splice3}. Using the recipe in Subsection~\ref{splicing}, one can find the additional weights as $-2$ and $-1$, which correspond to the determinants of the $E_7$ graph and the linear graph with vertices $-2, -1$ and $-3$, respectively. Next, we apply blow down operations six many times and we obtain the last picture in Figure~\ref{fig:splice3}. In Figure~\ref{fig:splice2}, we first draw the surgery diagram of the last plumbing graph appeared in Figure~\ref{fig:splice3}. Applying blow down operations seven many times more, we get the surgery diagram of $S^3_{-1} (T(2,3) \# \overline{T(2,3)})$, as expected. Our argument and procedure can be simply generalized to all values of $n$, so they are left to readers as exercises. In the general case, one can work with positive-definite plumbing graphs to avoid sign ambiguities of additional vertices appeared after the splicing operation.

\begin{figure}[htbp]
\begin{center}
\includegraphics[width=0.65\columnwidth]{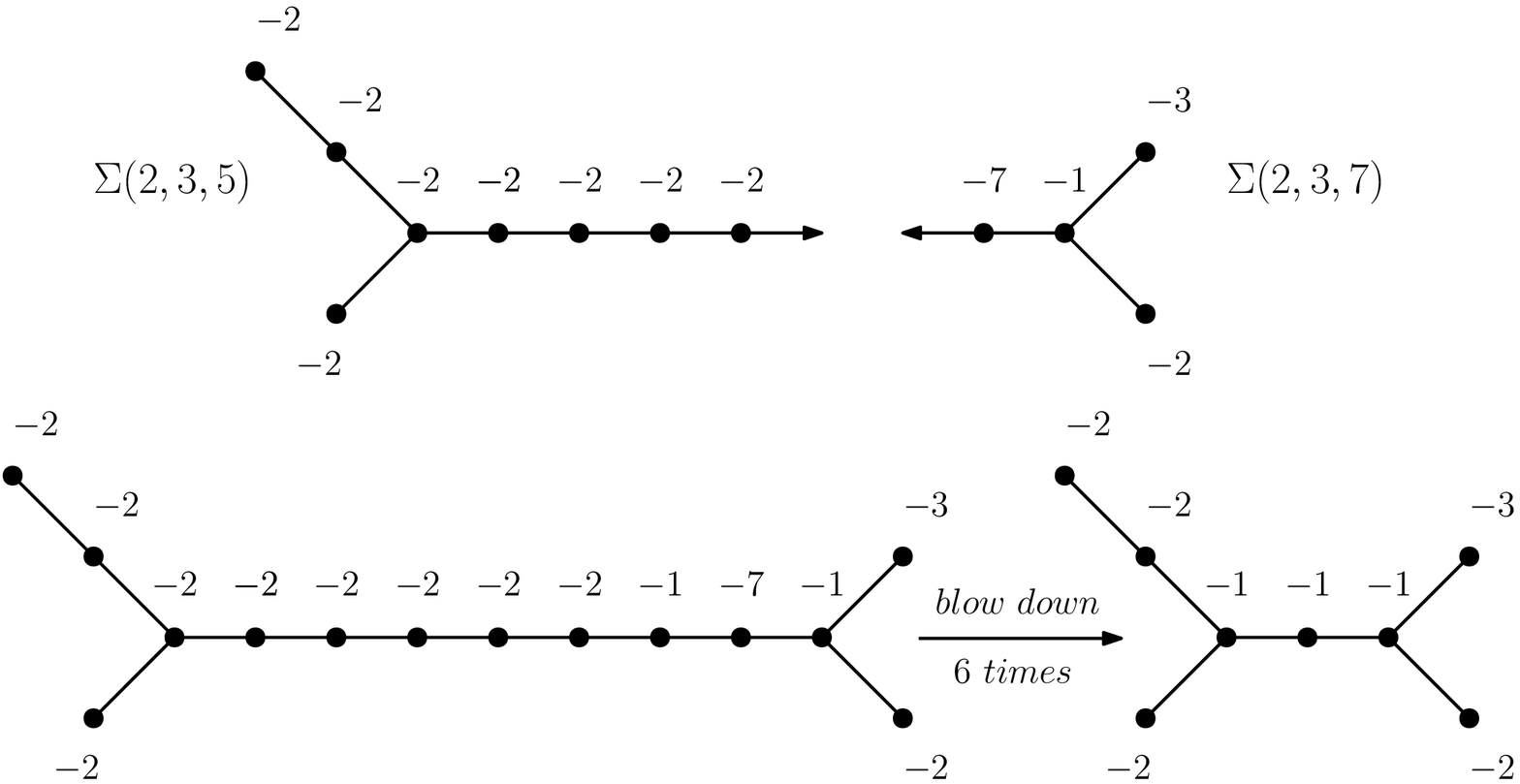}
\end{center}
\caption{The splice graph of $\Sigma(2,3,5) \tensor[_{K(5)}]{\bowtie}{_{K(7)}} \Sigma(2,3,7)$.}
\label{fig:splice3} 
\end{figure}

\begin{figure}[htbp]
\begin{center}
\includegraphics[width=0.65\columnwidth]{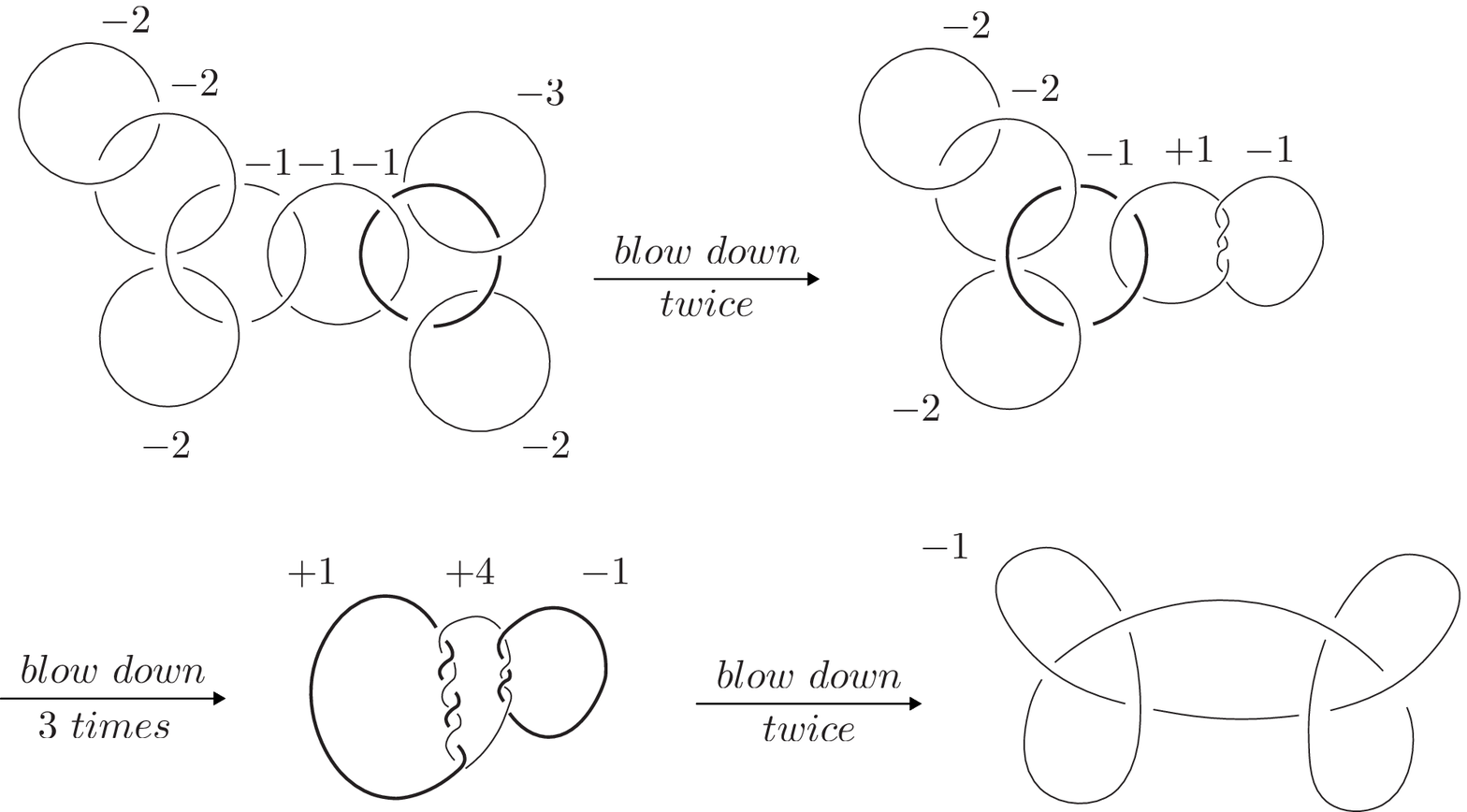}
\end{center}
\caption{A diffeomorphism between $\Sigma(2,3,5) \tensor[_{K(5)}]{\bowtie}{_{K(7)}} \Sigma(2,3,7)$ and $S^3_{-1} (T(2,3) \# \overline{T(2,3)})$.}
\label{fig:splice2} 
\end{figure}

\begin{rem}
An alternative proof argument for Theorem~\ref{splice} can be given as follows.

\begin{itemize}[leftmargin=2em]
\item One can prove our initial claim in full generality by using $3$-dimensional techniques of Fukuhara and Maruyama \cite[Pg.~285-286]{FM88}. They provided a surgery formula for the splicing operation of homology spheres. Their observation is a generalization of Gordon's previous arguments in \cite{G75} and \cite[Pg.~700-701]{G83}. Their useful observation was recently reproved by Karakurt, Lidman, and Tweedy by using Kirby calculus \cite[Lemma~2.1]{KLT21}. Thus we have
\begin{enumerate}

\item $\Sigma_1 \tensor[_{K_1}]{\bowtie}{_{K_2}} \Sigma_2 \approx S^3_{-1}(T(n+1,n+2)) \tensor[_{K_1}]{\bowtie}{_{K_2}} S^3_{-1}(\overline{T(n+1,n+2)})$, and
\item $S^3_{-1}(T(n+1,n+2)) \tensor[_{K_1}]{\bowtie}{_{K_2}} S^3_{-1}(\overline{T(n+1,n+2)}) \approx S^3_{-1} (T(n+1,n+2) \# \overline{T(n+1,n+2)})$.
\end{enumerate}

\item  A famous result of Gordon indicates that a homology sphere obtained by $(\pm 1)$-surgery along a slice knot in $S^3$ bounds a contractible $4$-manifold \cite[Theorem~3]{G75}.
\end{itemize}
\end{rem}

\end{proof}

\section{Further Directions}
\label{further}

We state some problems and questions for possible further research directions engaging with Kirby-Ramanujam spheres and their modifications. 

\begin{introques} Let $X$ be a homology plane and $\partial X$ be a Kirby-Ramanujam sphere bounding a Mazur or Po\'enaru manifold $W$. Does $W$ always admit a complex structure? If yes, is there a complex structure such that $W$ is biholomorphic to $X$?
\end{introques}

In a similar vein to the concept of Kirby-Ramanujam spheres, we introduce the objects called \emph{Ramanujam spheres}. They are non-trivial homology spheres bounding homology planes. In contrast to Theorem \ref{splice}, we have a certain constraint due to the article of Orevkov \cite{O97} in which he proved that Brieskorn spheres as well as Seifert fibered spheres cannot be Ramanujam spheres. 

\begin{introprob} 
Which homology spheres arise as Kirby-Ramanujam spheres or Ramanujam spheres?
\end{introprob}

In order to compare our objects, we may raise the following naive question:

\begin{introprob} 
Are Ramanujam spheres always Kirby-Ramanujam spheres?
\end{introprob}

A homology plane $X$ is known to be affine \cite{Fuj82}, see also \cite[Lemma~2.1]{Z98}. Therefore, $\partial X$ admits a contact structure. However, there are examples of Mazur manifolds with Brieskorn sphere boundaries such that they cannot admit Stein structures due to Mark and Tosun \cite{MT18}.  

\begin{introques} 
Is there any example of a Kirby-Ramanujam sphere that bounds an Akbulut cork?
\end{introques}

Since there are several examples of homology planes that are known to be non-contractible \cite{MM01, AA21}, we curiously ask the following question. Compare with \cite[Problem~G]{S22}.

\begin{introques} Does a Kirby-Ramanujam sphere bound a non-contractible homology plane?
\end{introques}

It would be interesting to study Ramanujam spheres by using the certain invariants of $3$-manifolds \cite{OS03, AN08, GM21, AM22}. Note that some these invariants are only defined for specific $3$-manifolds. Hence we inquire:
\begin{introques} 
Is it possible to extend/compute all these types of invariants for Ramanujam spheres?
\end{introques}

If the answer is positive, then one can pursue to extend invariants to $\mathbb{Q}$-homology planes. Similarly, they are affine complex smooth surfaces having the same $\mathbb{Q}$-homology as $\mathbb{C}^2$. For the basic properties and the partial classification of $\mathbb{Q}$-homology planes, see \cite{MM01, P21}.

Using the notion of bad vertices in \cite{AN05}, one can show that the graphs $X(n)$ and $Z(n)$ in Figure~\ref{fig:all} have two bad vertices. One can obtain rational singularities by reducing the weight of these vertices. It may be interesting if one could extend and compute the analytic lattice cohomology \cite{AN21} to graphs of type $2$ in the sense of \cite{OSS14} and relate it to analytic invariants of the algebraic surface.

In \cite[Chapter 14]{Sav99}, the representation varieties of Seifert fibered spheres were discussed. In contrast to the Orekov's theorem \cite{O97}, there exist homology planes of log-Kodaira dimension $1$ with the fundamental group isomorphic to that of a Seifert fibered manifold. Also, there is a surjective homorphism $\pi_1(\partial X)\to \pi_1(X)$ for a homology plane. Indeed, consider a smooth compactification $\bar{X}$ of $X$. It is known that $\bar{X}$ is rational, and hence simply-connected. In this case, the fundamental group $\pi_1(X)$ is generated by meridians around the irreducible components of $D$, see \cite[Section 5.6]{S07}. These meridians can be included in the generators of the boundary of the homology plane $\pi_1(\partial U)$. For the explicit description, see \cite{AA21}.

\begin{introprob} For a non-contractible homology plane $X$, study $\Hom(\pi_1(X), \SU(2))$ the representation variety of the Ramanujam sphere $\partial X$ as a subvariety of the representation variety $\Hom(\pi_1(\partial X), \SU(2))$ of $X$.
\end{introprob}

\begin{introques} Which invariants or representations of a Ramanujam sphere $\partial X$ can be extended to a homology plane $X$?
\end{introques}

\bibliography{homologyplanesandmazurmanifolds}
\bibliographystyle{amsalpha}

\end{document}